\documentclass[11pt,a4paper]{amsart}
\usepackage[utf8x]{inputenc} 
\usepackage{latexsym}
\usepackage{color,graphicx,shortvrb}
\usepackage{amsmath, amssymb}
\usepackage{amsfonts}
\usepackage[colorlinks, bookmarks=true]{hyperref}
\newtheorem{theorem}{Theorem}[section]
\newtheorem{lemma}[theorem]{Lemma}
\newtheorem{proposition}[theorem]{Proposition}
\newtheorem{corollary}[theorem]{Corollary}

\theoremstyle{definition}
\newtheorem{definition}[theorem]{Definition}

\title{Polynomial maps on the monoid of words}
\author[J. M. Almira]{Jose María Almira}
\address{J. M. Almira, Depto. Ingenier\'{\i}a y Tecnolog\'{\i}a de Computadores, Universidad de Murcia, 30100 Murcia, SPAIN}
\email{jmalmira@um.es}

\begin{document}
\begin{abstract}
 We briefly visit the theory of polynomial and semipolynomial maps defined on an arbitrary monoid, with range a commutative group. Then we characterize the space $\mathcal{P}(S,\mathbb{C})$ of polynomial maps $f:S\to \mathbb{C}$, where $S=\mathcal{A}^*$ is the monoid of words based on an arbitrary alphabet $\mathcal{A}$ under concatenation, and we prove that if there exists a monoid $S$ such that $\mathcal{SP}(S,\mathbb{C})\neq \mathcal{P}(S,\mathbb{C})$, then also $\mathcal{SP}(\mathcal{A}^*,\mathbb{C})\neq\mathcal{P}(\mathcal{A}^*,\mathbb{C})$ for a certain alphabet $\mathcal{A}$.  We propose as an open problem to prove or disprove that  $\mathcal{SP}(\mathcal{A}^*,\mathbb{C})=\mathcal{P}(\mathcal{A}^*,\mathbb{C})$ for arbitrary alphabets $\mathcal{A}$. Finally, we apply our algebraic framework to combinatorial objects by proving that permutation pattern functions are polynomial maps. 
  Our results are motivated by previous work of Shulman \cite{Shulman_2025}, \cite{Shulman_2019}. 

\end{abstract}
\maketitle
\markboth{J. M. Almira}{Polynomial maps}



\section{Polynomial maps with range a commutative group}

Throughout  this paper, we assume that $(S,\cdot,e)$ is a monoid and $(G,+,0)$ is a commutative group, and consider the operators:
\begin{align*}
L_{s}:G^S\to G^S; & \quad L_s(f)(t)=f(st)-f(t),\\
R_{s}:G^S\to G^S; & \quad R_s(f)(t)=f(ts)-f(t),
\end{align*}
which are called the left and right difference operators, respectively. In this noncommutative setting, polynomial maps are defined  as follows:  The constant functions $f: S\to G$ are polynomial maps. 
If the constant is equal to $0$, then its functional degree is $-\infty$. 
If the constant is different from $0$, then its functional degree is $0$. 
Inductively, for each $n\ge 1$, the function $f:S\to G$ is a polynomial map of functional degree $\leq n$ if and only if for each $s\in S$, the functions $R_s(f), L_s(f)$ are both polynomial maps of functional degree $\leq n-1$. 
The minimal $n$  with this property is called the \emph{functional degree} of $f$.  If we only use the operator $L_s$, we speak of left-polynomial maps and left-functional degree; if we only use the operator $R_s$, we speak of right-polynomial maps and right-functional degree.  

We denote by $\mathcal{P}_n(S,G)$ the set of all polynomial maps $f\in G^S$ with functional degree $\leq n$.   
For left-polynomials maps of left-functional degree $\leq n$, we use the notation $\mathcal{P}_n^{L}(S,G)$, while for right-polynomial maps of right-functional degree $\leq n$, we use the notation $\mathcal{P}_n^R(S,G)$. Obviously, we have 
\[
\mathcal{P}_n(S,G)\subseteq
\mathcal{P}_n^L(S,G) \cap 
\mathcal{P}_n^R(S,G). 
\]
\subsection{Cauchy $n$-balanced and Aichinger's functional equations } 
In \cite{Shulman_2025}, it was essentially proved that (see also \cite[Chapter 6]{AlmiraHu}) 
\begin{equation}\label{allpolequal}
\mathcal{P}_n(S,G)
=\mathcal{P}_n^L(S,G)
=\mathcal{P}_n^R(S,G)
\end{equation}
so that the distinction between left-, right- and just polynomial maps is unnecessary when $G$ is commutative. We include a draft of the proof, for the sake of completeness. To prove the result, it is necessary to introduce the Cauchy $n$-balanced functional equation:
\begin{equation}\label{balanced}
\sum_{T\subseteq\{1,\dots,n\}} (-1)^{n-|T|}  f(\prod_{i\in T} x_i)=0 \quad \text{ {\rm for all }} x_1,\dots,x_n\in S,
\end{equation}
(where, for each subset $T\subseteq \{1,\dots,n\}$, the factors in the product $\prod_{i\in T} x_i$ appear in the natural increasing order of the indices, and for the empty set $T=\emptyset$ we put $\prod_{i\in \emptyset} x_i=e$)
and prove that $f\in G^S$ is a right-polynomial map of right-functional degree $<n$ if and only if it solves \eqref{balanced}, and the same happens with left-polynomial maps of left-functional degree $<n$, so that  $\mathcal{P}_n(S,G)\subseteq  \mathcal{P}_n^L(S,G)=\mathcal{P}_n^R(S,G)$.

Equation \eqref{balanced} appears quite naturally, indeed: for functions $f\in G^S$, we have that  $R_h=(\tau_h^R-I)$ and $L_h=(\tau_h^L-I)$, where $\tau_h^Rf(x)=f(xh)$, $\tau_h^Lf(x)=f(hx)$, so that 
\[
R_{x_1}\cdots R_{x_n} f = \prod_{i=1}^n (\tau^R_{x_i}-I) f = \sum_{T\subseteq\{1,\dots,n\}} (-1)^{n-|T|} \tau^R_{\prod_{i\in T} x_i} f,
\]
and 
\[
L_{x_1}\cdots L_{x_n} f = \prod_{i=1}^n (\tau^L_{x_i}-I) f = \sum_{T\subseteq\{1,\dots,n\}} (-1)^{n-|T|} \tau^L_{\prod_{i\in T} x_i} f.
\]
Thus, $f$ is a right-polynomial map of right-functional degree $<n$ if and only if
\[
\sum_{T\subseteq\{1,\dots,n\}} (-1)^{n-|T|}  f(x\prod_{i\in T} x_i)=0 \quad \text{ {\rm for all }} x_1,\dots,x_n,x\in S 
\]
and, evaluating $e\in S$ (the unit of $S$), we get \eqref{balanced}.  Moreover, if $R_{x_1}\cdots R_{x_n}f(e)=0$ for all $x_1,\dots,x_n\in S$, then 
\begin{eqnarray*}
   0 &=& R_{x_1}R_{x_2}\cdots R_{x_n}f(e) = R_{x_1}(R_{x_2}\cdots R_{x_n}f)(e)\\
     &=& R_{x_2}\cdots R_{x_n}f(x_1)-R_{x_2}\cdots R_{x_n}f(e) \quad \text{{\rm for all }} x_1\in S
\end{eqnarray*} 
which means that  $R_{x_2}\cdots R_{x_n}f$ is a constant function and, henceforth, $R_{x_1}R_{x_2}\cdots R_{x_n}f$ vanishes identically. Thus, $f$ is a right-polynomial map of right-functional degree $<n$ if and only if \eqref{balanced} holds. The same argument, applied to left-polynomial maps yields that $f\in G^S$ is a left-polynomial map of left-functional degree $<n$ if and only if 
\[
\sum_{T\subseteq\{1,\dots,n\}} (-1)^{n-|T|}  f(\left(\prod_{i\in T} x_i\right)x)=0 \quad \text{ {\rm for all }} x_1,\dots,x_n,x\in S.
\]
and, taking $x=e$, we get again the functional equation \eqref{balanced} and we can repeat the same argument with left-polynomial maps.  Thus, the sets $\mathcal{P}_n^L(S,G)$ and $\mathcal{P}_n^R(S,G)$ coincide, and they define precisely the solutions set of Cauchy $(n+1)$-balanced functional equation.

On the other hand, the commutativity of $G$ implies that, for all $s_1,s_2\in S$,  $L_{s_1}R_{s_2}=R_{s_2}L_{s_1}$ (it is a direct computation). Consequently, if $\{s_1,\dots,s_n\}\subset S$ and $D_i\in\{R_{s_i},L_{s_i}\}$ for $i=1,\dots,n$, then 
    \[
    D_1D_2\cdots D_n=L_{s_{i_1}}L_{s_{i_2}}\cdots L_{s_{i_k}}R_{s_{j_1}}R_{s_{j_2}}\cdots R_{s_{j_{n-k}}} 
    \]
    for a certain choice of indices $i_\ell$, $j_t$ and $k$ such that: $$\{i_1,\dots,i_k,j_1,\dots,j_{n-k}\}=\{1,2,\dots,n\}.$$
  Thus, the function $f:S\to G$ is a polynomial map of functional degree $<n$ if and only if $f$ solves, for each $k\in \{0,\dots,n\}$,  the system of equations: 
     \begin{equation}\label{defpol}
         L_{s_{1}}L_{s_{2}}\cdots L_{s_{k}}R_{h_{1}}R_{h_{2}}\cdots R_{h_{n-k}} f(x)=0 \quad s_1,\dots,s_k,h_1,\dots,h_{n-k},x\in S.
     \end{equation}
Then, using that $\mathcal{P}_m^L(S,G)=\mathcal{P}_m^R(S,G)$ for all $m\in\mathbb{N}$, it is possible to prove that, in fact, $f\in \mathcal{P}_{n-1}(S,G)$ as soon as it solves \eqref{defpol} for at least one $k\in\{0,\dots,n\}$. Indeed,  let us assume that $f$ solves \eqref{defpol} for a given $k\in \{0,1,\dots,n\}$. Then $R_{h_{1}}R_{h_{2}}\cdots R_{h_{n-k}} f\in \mathcal{P}_k^L(S,G)=\mathcal{P}_k^R(S,G)$ for all $h_1,\dots,h_{n-k}\in S$. Hence 
    \[
    R_{s_{1}}R_{s_{2}}\cdots R_{s_{k}}R_{h_{1}}R_{h_{2}}\cdots R_{h_{n-k}} f(x)=0 \quad s_1,\dots,s_k,h_1,\dots,h_{n-k},x\in S
    \]
    and $f\in \mathcal{P}_{n-1}^R(S,G)= \mathcal{P}_{n-1}^L(S,G)$. In particular, $f$ solves the equation \eqref{defpol} for $0,k$ and $n$.  Let us now select $k^*\in\{1,\dots,n-1\}$, $k^*\neq k$. Then $f\in \mathcal{P}_{n-1}^R(S,G)$ implies that 
    \[
R_{s_{1}}R_{s_{2}}\cdots R_{s_{k^*}}R_{h_{1}}R_{h_{2}}\cdots R_{h_{n-k^*}} f(x)=0 \quad s_1,\dots,s_{k^*},h_1,\dots,h_{n-k^*},x\in S
    \]
    Hence $R_{h_{1}}R_{h_{2}}\cdots R_{h_{n-k^*}} f\in \mathcal{P}_{k^*}^R(S,G)=\mathcal{P}_{k^*}^L(S,G)$, and $f$ solves the equation 
    \[
    L_{s_{1}}L_{s_{2}}\cdots L_{s_{k^*}}R_{h_{1}}R_{h_{2}}\cdots R_{h_{n-k^*}} f(x)=0 \quad s_1,\dots,s_{k^*},h_1,\dots,h_{n-k^*},x\in S.
    \]
   Thus, $f$ solves \eqref{defpol} for all $k\in\{0,\dots,n\}$, which means that  $f\in \mathcal{P}_{n-1}(S,G)$. In particular, the identity \eqref{allpolequal} holds true for all $n$. 

Another important characterization of the elements of $\mathcal{P}_{n-1}(S,G)$ is as follows: $\operatorname{fdeg}(f)<n$ if and only if $f$ solves Aichinger's equation:
\begin{equation}
        \label{61_Aichinger} f(x_1x_2\cdots x_n)=\sum_{i=1}^nF_i(x_1,\dots,\widehat{x_i},\dots,x_n) \quad x_1,\dots,x_n\in S.
    \end{equation}
Indeed, if $\operatorname{fdeg}(f)<n$, then $f$ solves the equation \eqref{balanced}, that can be rearranged as: 
\begin{equation}\label{balanced2}
f(x_1x_2\cdots x_{n})=\sum_{T\subsetneq\{1,\dots,n\}} (-1)^{n+1-|T|}  f(\prod_{k\in T} x_k) = \sum_{i=1}^n\left(\sum_{T\subsetneq\{1,\dots,n\},i\not\in T} (-1)^{n+1-|T|}  f(\prod_{k\in T} x_k)\right) 
\end{equation}
which means that $f$ also solves \eqref{61_Aichinger} with 
\[
F_i(x_1,\dots,\widehat{x_i},\dots,x_n)= \sum_{T\subsetneq\{1,\dots,n\},i\not\in T} (-1)^{n+1-|T|}  f(\prod_{k\in T} x_k). 
\]
The fact that solutions of \eqref{61_Aichinger} are polynomials of functional degree $<n$ is a particular case of \cite[Theorem 2.3]{Shulman_2025}. Moreover, the same result can be proved just adapting the arguments of \cite[Theorem 2.1]{Almira_2023} to this noncommutative setting. 

\subsection{Functional degree of the pointwise product} \label{pointwiseproduct}
Both $n$-balanced Cauchy and Aichinger's functional equations are useful to prove many results about polynomial maps. In particular, we use them to give two distinct proofs of the inequality
\begin{equation}\label{fdegproduct}
    \operatorname{fdeg}(f\cdot g)\leq \operatorname{fdeg}(f)+\operatorname{fdeg}(g) \quad \text{ whenever } f,g\in \mathfrak{R}^S,
\end{equation}
where $(\mathfrak{R},+,\cdot)$ is any commutative ring, and $(f\cdot g)(x)=f(x)\cdot g(x)$ is the pointwise product of $f,g:(S,\cdot)\to (\mathfrak{R},+)$.

\begin{lemma}\label{lem:6.combinatorial_split_proof}
Let $S$ be a monoid and $\mathfrak{R}$ a commutative ring. For any functions $f, g \in \mathfrak{R}^S$ and a finite index set $\underline{n} = \{1, \dots, n\}$, the following identity holds:
\begin{align}
\sum_{T \subseteq \underline{n}} (-1)^{n-|T|} f\left(\prod_{i\in T} x_i\right) g\left(\prod_{i\in T} x_i\right) &= \nonumber \\
\sum_{\substack{A \cup B = \underline{n}}} \left( \sum_{A' \subseteq A} (-1)^{|A|-|A'|} f\left(\prod_{i \in A' } x_i\right) \right) &\cdot \left( \sum_{B' \subseteq B} (-1)^{|B|-|B'|} g\left(\prod_{j \in  B'} x_j\right) \right), \label{eq:combinatorial_split_explicit}
\end{align}
where $A,B\subseteq \underline{n}$ satisfy $A \cup B = \underline{n}$, and the empty product $\prod_{i \in \emptyset} x_i$ is defined to be $e \in S$.
\end{lemma}
\begin{proof}
Let us evaluate $R_x(f \cdot g)$ at an arbitrary point $y \in S$:
\[
R_x(fg)(y) = f(yx)g(yx) - f(y)g(y).
\]
Adding and subtracting $f(y)g(yx)$, we factor this expression as:
\[
R_x(fg)(y) = [f(yx) - f(y)]g(yx) + f(y)[g(yx) - g(y)] = (R_x f)(y) \cdot (\tau_x g)(y) + f(y) \cdot (R_x g)(y).
\]
Substituting $\tau_x = \operatorname{I} + R_x$ into the first term yields the fundamental $1$-variable discrete Leibniz rule:
\begin{equation}\label{eq:leibniz_1var}
R_x(fg) = (R_x f)g + f(R_x g) + (R_x f)(R_x g).
\end{equation}
For any subset of indices $I = \{i_1, \dots, i_k\} \subseteq \underline{n}$ with $i_1 < \dots < i_k$, define the compound right difference operator $R_I = R_{x_{i_1}} \cdots R_{x_{i_k}}$, setting $R_\emptyset = \operatorname{I}$.

Applying identity~\eqref{eq:leibniz_1var} successively for $i = 1, \dots, n$ in increasing order, each operator $R_{x_i}$ acts on the product $fg$ by being assigned to $f$ alone ($i \in A \setminus B$), to $g$ alone ($i \in B \setminus A$), or to both $f$ and $g$ simultaneously ($i \in A \cap B$). Collecting all choices over $i = 1, \dots, n$ yields pairs of subsets $(A, B)$ satisfying $A \cup B = \underline{n}$, giving the operator identity:
\begin{equation}\label{eq:leibniz_multivar}
R_{x_1} \cdots R_{x_n}(fg) = \sum_{A \cup B = \underline{n}} (R_A f) \cdot (R_B g).
\end{equation}
We now evaluate both sides of operator identity~\eqref{eq:leibniz_multivar} at $e$: On the left-hand side, expanding $R_{x_1} \cdots R_{x_n} = \prod_{i=1}^n (\tau_{x_i} - \operatorname{I})$ gives:
\[
R_{x_1} \cdots R_{x_n}(fg)(e) = \sum_{T \subseteq \underline{n}} (-1)^{n-|T|} (fg)\left(\prod_{i \in T} x_i\right) = \sum_{T \subseteq \underline{n}} (-1)^{n-|T|} f\left(\prod_{i \in T} x_i\right) g\left(\prod_{i \in T} x_i\right).
\]
On the right-hand side, for each pair $(A, B)$ with $A \cup B = \underline{n}$, expanding $R_A = \prod_{i \in A} (\tau_{x_i} - \operatorname{I})$ at $e$ yields:
\[
R_A f(e) = \sum_{A' \subseteq A} (-1)^{|A|-|A'|} f\left(\prod_{i \in A'} x_i\right),
\]
and similarly $R_B g(e) = \sum_{B' \subseteq B} (-1)^{|B|-|B'|} g\left(\prod_{j \in B'} x_j\right)$.

Substituting these expansions into the right-hand side of~\eqref{eq:leibniz_multivar} yields the desired right-hand side of Lemma~\ref{lem:6.combinatorial_split_proof}, where the outer sum runs over all pairs $(A, B)$ with $A \cup B = \underline{n}$. This completes the proof.
\end{proof}

\noindent \textbf{Proof of \eqref{fdegproduct} based on Cauchy $n$-balanced functional equation \eqref{balanced}: } 
Let $\operatorname{fdeg}(f) = d_1$ and $\operatorname{fdeg}(g) = d_2$, and define $n = d_1 + d_2 + 1$. To establish that $\operatorname{fdeg}(f \cdot g) \le n - 1$, it suffices to demonstrate that $f \cdot g$ satisfies the $n$-balanced Cauchy equation~\eqref{balanced}. That is, we must show that for any $x_1, \dots, x_n \in S$:
\[
\sum_{T\subseteq\underline{n}} (-1)^{n-|T|} (f \cdot g)\left(\prod_{i\in T} x_i\right) = 0,
\]
where, for each subset $T\subseteq \underline{n}$, the factors in the product $\prod_{i\in T} x_i$ appear in the natural increasing order of the indices, and for the empty set $T=\emptyset$ we put $\prod_{i\in \emptyset} x_i=e$. We follow this identical index-ordering convention for all sums indexed over subsets in what follows.  By definition of the pointwise product, the expression inside the summation can be written as:
\[
(f \cdot g)\left(\prod_{i\in T} x_i\right) = f\left(\prod_{i\in T} x_i\right) \cdot g\left(\prod_{i\in T} x_i\right).
\]
To evaluate the sum, we use Lemma \ref{lem:6.combinatorial_split_proof}, which claims that 
\begin{align}
\sum_{T \subseteq \underline{n}} (-1)^{n-|T|} f\left(\prod_{i\in T} x_i\right) g\left(\prod_{i\in T} x_i\right) &= \nonumber \\
\sum_{A \cup B = \underline{n}} \left( \sum_{A' \subseteq A} (-1)^{|A|-|A'|} f\left(\prod_{i \in A'} x_i\right) \right) &\cdot \left( \sum_{B' \subseteq B} (-1)^{|B|-|B'|} g\left(\prod_{j \in B'} x_j\right) \right). \label{eq:combinatorial_split}
\end{align}

Let us analyze the structure of the internal sums on the right-hand side of equation~\eqref{eq:combinatorial_split} for a fixed pair $(A, B)$ satisfying $A\cup B=\underline{n}$. Obviously, $|A| + |B| \geq |\underline{n}| = n = d_1 + d_2 + 1$.Thus, either $|A| \ge d_1 + 1$, or $|B| \ge d_2 + 1$.

Suppose that $|A| \ge d_1 + 1$. The inner alternating sum corresponding to the function $f$ 
corresponds precisely to an $|A|$-balanced Cauchy operator applied to $f$. Since $\operatorname{fdeg}(f) = d_1$ and $|A| \ge d_1 + 1$, the characterization of polynomial maps by the balanced Cauchy equation guarantees that this variation vanishes identically:
\[
\sum_{A' \subseteq A} (-1)^{|A|-|A'|} f\left(\prod_{i \in A' } x_i\right) = 0.
\]
Symmetrically, if $|B| \ge d_2 + 1$, the inner alternating sum corresponding to the function $g$ represents a $|B|$-balanced Cauchy operator applied to $g$. Since $\operatorname{fdeg}(g) = d_2$, this operator vanishes identically:
\[
\sum_{B' \subseteq B} (-1)^{|B|-|B'|} g\left(\prod_{j \in  B'} x_j\right) = 0.
\]
Consequently, for every such pair $(A, B)$, at least one of the two multiplying factors in the product on the right-hand side of equation~\eqref{eq:combinatorial_split} is equal to zero.  Because every term in the summation vanishes, the entire sum collapses to zero:
\[
\sum_{T\subseteq\{1,\dots,n\}} (-1)^{n-|T|} (f \cdot g)\left(\prod_{i\in T} x_i\right) = 0.
\]
Thus, the pointwise product $f \cdot g$ satisfies the $n$-balanced Cauchy equation identically. We conclude that $f \cdot g$ is a polynomial map of functional degree at most $n - 1$:
\[
\operatorname{fdeg}(f \cdot g) \le n - 1 = d_1 + d_2 = \operatorname{fdeg}(f) + \operatorname{fdeg}(g),
\]
which completes the proof.

\medskip

For a proof based on Aichinger's equation, we first show the following: 

\begin{lemma}\label{decompositionrightoperators}
    For any function $f \in G^S$ and elements $x_1, \dots, x_n \in S$, the value of $f$ at the product $x_1 \cdots x_n$ can be expanded via the right-difference operators $R_x$ as:
\begin{equation}\label{eq:right_diff_expansion}
f(x_1 x_2 \cdots x_n) = \sum_{I \subseteq \{1, \dots, n\}} R_{I} f(1).
\end{equation}
where 
$1$ is the identity element of $S$.
\end{lemma}

\begin{proof}
Recall that $R_y f(x) = f(xy) - f(x)$. So, by isolating the shifted term, we get
\[
f(xy) = f(x) + R_y f(x) = (\operatorname{I} + R_y)f(x),
\]
where $\operatorname{I}$ denotes the identity operator. We proceed to prove the general expansion identity~\eqref{eq:right_diff_expansion} by induction on the number of factors $n$.

For a single element $x_1 \in S$, evaluating the function at $x_1$ is equivalent to evaluating it at the right-shifted identity element $1 \cdot x_1$:
\[
f(x_1) = f(1 \cdot x_1) = f(1) + R_{x_1}f(1).
\]
The index set $\{1\}$ possesses exactly two subsets: the empty set $\emptyset$ and the singleton $\{1\}$. Expanding the right-hand side of equation~\eqref{eq:right_diff_expansion} under the convention that $R_{\emptyset} = \operatorname{I}$ yields:
\[
\sum_{I \subseteq \{1\}} R_I f(1) = R_{\emptyset}f(1) + R_{\{1\}}f(1) = f(1) + R_{x_1}f(1).
\]
Thus, the base case holds identically.

Assume that the identity holds for a product of $k$ elements. That is, for any choice of elements $x_1, \dots, x_k \in S$, we have:
\begin{equation}\label{eq:induction_hypothesis}
f(x_1 x_2 \cdots x_k) = \sum_{J \subseteq \{1, \dots, k\}} R_{J} f(1).
\end{equation}
Now, consider the evaluation of a product containing $k+1$ elements, written as $(x_1 \cdots x_k) \cdot x_{k+1}$. Applying the fundamental right-difference relation to strip off the terminal element $x_{k+1}$ gives:
\begin{equation}\label{eq:split_step}
f(x_1 \cdots x_k x_{k+1}) = f(x_1 \cdots x_k) + R_{x_{k+1}}f(x_1 \cdots x_k).
\end{equation}
We now apply the induction hypothesis~\eqref{eq:induction_hypothesis} directly to both terms on the right-hand side of equation~\eqref{eq:split_step}:
\begin{enumerate}
    \item The first term expands directly over the power set of $\{1, \dots, k\}$:
    \[
    f(x_1 \cdots x_k) = \sum_{J \subseteq \{1, \dots, k\}} R_{J} f(1).
    \]
    \item For the second term, we define an auxiliary function $g \in G^S$ by $g(x) = R_{x_{k+1}}f(x)$. Applying the induction hypothesis to the evaluation $g(x_1 \cdots x_k)$ reveals:
    \[
    R_{x_{k+1}}f(x_1 \cdots x_k) = g(x_1 \cdots x_k) = \sum_{J \subseteq \{1, \dots, k\}} R_{J} g(1) = \sum_{J \subseteq \{1, \dots, k\}} R_{J} \left(R_{x_{k+1}}f(1)\right).
    \]
    Since the index $k+1$ is strictly greater than any index contained in the subset $J$, the operation $R_{x_{k+1}}$ is placed correctly at the end of the chain, preserving the natural increasing order of indices:
    \[
    R_{J} \left(R_{x_{k+1}}f(1)\right) = R_{J \cup \{k+1\}} f(1).
    \]
\end{enumerate}
Substituting both independent subset sums back into equation~\eqref{eq:split_step} yields:
\[
f(x_1 \cdots x_{k+1}) = \sum_{J \subseteq \{1, \dots, k\}} R_{J} f(1) + \sum_{J \subseteq \{1, \dots, k\}} R_{J \cup \{k+1\}} f(1).
\]
Any arbitrary subset $I \subseteq \{1, \dots, k, k+1\}$ can be classified into one of two disjoint classes: it either completely excludes the terminal element $k+1$, or it includes it. 
\begin{itemize}
    \item The first summation ranges over all valid subsets $I$ that do not contain $k+1$ (where $I = J$).
    \item The second summation ranges over all valid subsets $I$ that must contain $k+1$ (where $I = J \cup \{k+1\}$).
\end{itemize}
Recombining these two disjoint collections of subsets reconstitutes the complete power set of $\{1, \dots, k+1\}$:
\[
f(x_1 x_2 \cdots x_{k+1}) = \sum_{I \subseteq \{1, \dots, k+1\}} R_{I} f(1).
\]
This proves the claim for all $n \ge 1$.

\end{proof}

\noindent \textbf{Proof of \eqref{fdegproduct} based on Aichinger's functional equation \eqref{61_Aichinger}: } 
Let $\operatorname{fdeg}(f) = d_1$ and $\operatorname{fdeg}(g) = d_2$. Define $n = d_1 + d_2 + 1$. To show that $\operatorname{fdeg}(f \cdot g) \le n - 1$, it suffices to show that the multi-variable evaluation $(f \cdot g)(x_1 x_2 \cdots x_n)$ can be decomposed into a sum of functions where each term omits at least one variable $x_i$.

Lemma \ref{decompositionrightoperators} informs us that for any function $h \in \mathfrak{R}^S$ and elements $x_1, \dots, x_n \in S$, the value of $h$ at the product $x_1 \cdots x_n$ can be expanded via the right-difference operators $R_x$ as:
\[
h(x_1 x_2 \cdots x_n) = \sum_{I \subseteq \{1, \dots, n\}} R_{I} h(1),
\]
where $R_{I}$ represents the composition of operators $R_{x_i}$ for $i \in I$ in increasing order of indices, and $1$ is the identity element of $S$. Applying this identity to the functions $f$ and $g$ individually yields:
\[
f(x_1 \cdots x_n) = \sum_{A \subseteq \{1, \dots, n\}} R_{A} f(1) \quad \text{and} \quad g(x_1 \cdots x_n) = \sum_{B \subseteq \{1, \dots, n\}} R_{B} g(1).
\]
Because $\mathfrak{R}$ is a commutative ring, the pointwise product of the evaluations satisfies:
\[
(f \cdot g)(x_1 \cdots x_n) = f(x_1 \cdots x_n) \cdot g(x_1 \cdots x_n) = \sum_{A, B \subseteq \{1, \dots, n\}} \left( R_{A} f(1) \cdot R_{B} g(1) \right).
\]
Now, let us examine the subsets of indices $A$ and $B$ in the summation. Since $\operatorname{fdeg}(f) = d_1$ and $\operatorname{fdeg}(g) = d_2$, the difference operator $R_{A} f$ vanishes identically if $|A| > d_1$, and $R_{B} g$ vanishes identically if $|B| > d_2$. Thus, the only non-zero terms in the sum occur when both $|A| \le d_1$ and $|B| \le d_2$. 

For any such surviving pair of subsets $(A, B)$, the cardinality of their union satisfies the bound:
\[
|A \cup B| \le |A| + |B| \le d_1 + d_2 = n - 1.
\]
Since the total number of available variables is $n$, the fact that $|A \cup B| \le n - 1$ guarantees that for every single non-zero term in the expansion, there exists at least one index $i \in \{1, \dots, n\}$ such that $i \notin A \cup B$. Consequently, neither $R_{A} f(1)$ nor $R_{B} g(1)$ depends on the variable $x_i$.

We can therefore partition the summation by grouping terms according to the first missing variable index. For each $i \in \{1, \dots, n\}$, let $\mathcal{M}_i$ be the collection of pairs $(A, B)$ such that $i \notin A \cup B$ and $i$ is the minimal index with this property. We define:
\[
F_i(x_1, \dots, \widehat{x_i}, \dots, x_n) = \sum_{(A,B) \in \mathcal{M}_i} \left( R_{A} f(1) \cdot R_{B} g(1) \right).
\]
By construction, each $F_i$ is a well-defined function of $n-1$ variables that does not depend on $x_i$. Summing over all possible missing index blocks gives:
\[
(f \cdot g)(x_1 x_2 \cdots x_n) = \sum_{i=1}^n F_i(x_1, \dots, \widehat{x_i}, \dots, x_n).
\]
This expansion matches the form of equation~\eqref{61_Aichinger} identically. Hence  $f \cdot g$ is a polynomial map of functional degree:
\[
\operatorname{fdeg}(f \cdot g) \le n - 1 = d_1 + d_2 = \operatorname{fdeg}(f) + \operatorname{fdeg}(g),
\]
which completes the proof.

\section{Semipolynomial maps with range a uniquely divisible commutative group }

A map $f:S\to G$ is a right-semipolynomial map of right-functional degree $\leq m$ if
\[
R_h^{m+1} (f)(x) = 0\quad \text{{\rm for all }} x,h\in S, 
\]
and it is a left-semipolynomial map of left-functional degree $\leq m$ if 
\[
L_h^{m+1} (f)(x) = 0\quad \text{{\rm for all }} x,h\in S. 
\]
Finally, $f$ is a semipolynomial map of functional degree $\leq m$ if 
\[
D_1D_2\cdots D_{m+1}f(x)=0 \quad \text{{\rm for all }} x,h\in S, \text{{\rm whenever }} D_i\in \{R_h,L_h\}, \ i=1,\dots,m+1.
\]
We write $\mathcal{SP}_m^R(S,G)$, $\mathcal{SP}_m^L(S,G)$, and $\mathcal{SP}_m(S,G)$ for these sets. Also, we use the following notations: $\mathcal{P}(S,G)=\bigcup_{m\ge 0}\mathcal{P}_m(S,G)$, $\mathcal{SP}^R(S,G)=\bigcup_{m\ge 0}{SP}^R_m(S,G)$, $\mathcal{SP}^L(S,G)=\bigcup_{m\ge 0}{SP}^L_m(S,G)$, and 
$\mathcal{SP}(S,G)=\bigcup_{m\ge 0}{SP}_m(S,G)$.

Clearly $$\mathcal{P}_m(S,G)\subseteq \mathcal{SP}_m(S,G)\subseteq \mathcal{SP}_m^R(S,G)\cap \mathcal{SP}_m^L(S,G) \subseteq \mathcal{SP}_m^R(S,G)\cup \mathcal{SP}_m^L(S,G)$$ for every $m$. In general, the reverse inclusions do not hold.  In \cite{Shulman_2019} the class $\mathcal{CS}$ of commensurable semigroups (i.e. the semigroups $S$ satisfying that $gS=Sg$ for all $g\in S$) was introduced and, it was used to prove that, if $S\in \mathcal{CS}$ and $G$ is a uniquely divisible commutative group (equivalently, $G$ is the additive group of a
$\mathbb Q$-vector space),  then $\mathcal{SP}_m^R(S,G)= \mathcal{SP}_m^L(S,G)$ and $\mathcal{SP}_m^R(S,G)\subseteq \mathcal{P}_{(m+1)\cdot z(2m+1)}^R(S,G)$ for all $m$, where
$z(\cdot)$ is the function arising from Zelmanov's theorem on Engel Lie algebras. In particular, 
 $\mathcal{P}(S,G)=\mathcal{SP}(S,G)=\mathcal{SP}^L(S,G)=\mathcal{SP}^R(S,G)$. This result, that I call Shulman's Theorem, has a difficult proof and is based on a deep theorem of Zelmanov (whose statement as well as a detailed proof can be found in \cite{Zelmanov_1987,Zelmanov_1990}) that has been useful in different branches of mathematics. It is, perhaps, the deepest known theorem about the structure of polynomial maps defined on arbitrary semigroups. A detailed proof of Shulman's theorem can also be found in \cite[Chapter 6]{AlmiraHu}. 

Condition $S\in\mathcal{CS}$ is sufficient but not necessary to guarantee that $\mathcal{SP}(S,G)=\mathcal{P}(S,G)$. For example, if we set $S = (M_2(\mathbb{R}),\cdot)$, the semigroup of all $2\times 2$ real matrices under multiplication, then $S\not\in \mathcal{CS}$, and it is not difficult to prove that $\mathcal{SP}(M_2(\mathbb R),\mathbb{C})
=
\mathcal{P}(M_2(\mathbb R),\mathbb{C})
=
\{\text{constant functions}\}$.

Indeed, consider
\[
g =
\begin{pmatrix}
1 & 0\\
0 & 0
\end{pmatrix}.
\]
Then
\[
gS = \left\{
\begin{pmatrix}
a & b\\
0 & 0
\end{pmatrix}
:a,b\in\mathbb{R}\right\}, \qquad
Sg = \left\{
\begin{pmatrix}
a & 0\\
b & 0
\end{pmatrix}:a,b\in\mathbb{R}
\right\}.
\]
These sets are different, so $gS \neq Sg$, and therefore $S \notin \mathcal{CS}$.

Let us now study the sets of polynomials and semipolynomials $f:S\to\mathbb{C}$. Set 
\[
0_2=
\begin{pmatrix}
0&0\\
0&0
\end{pmatrix}.
\]
Clearly, \(0_2\) is an absorbing element of \(S\): $x0_2=0_2$  for all $x\in S$. Hence
\[
\Delta_{0_2}f(x)
=
f(x0_2)-f(x)
=
f(0_2)-f(x).
\]
Applying \(\Delta_{0_2}\) once more gives
\[
\Delta_{0_2}^2f(x)
=
\Delta_{0_2}\bigl(f(0_2)-f(x)\bigr)
=
-\bigl(f(0_2)-f(x)\bigr)=
-\Delta_{0_2}f(x).
\]
By induction,
\[
\Delta_{0_2}^{\,k}f
=
(-1)^{k-1}\Delta_{0_2}f,
\qquad k\ge1.
\]

Suppose first that \(f\) is a semipolynomial of degree at most \(n\). Then
\[
0
=
\Delta_{0_2}^{\,n+1}f
=
(-1)^n\Delta_{0_2}f,
\]
and therefore
\[
\Delta_{0_2}f=0.
\]
Thus
\[
f(0_2)-f(x)=0
\qquad\text{for all }x\in S,
\]
which implies
\[
f(x)=f(0_2)
\qquad\text{for all }x\in S.
\]
Hence \(f\) is constant. But all constant functions are polynomials, and all polynomials are semipolynomials. Hence 
\[
\mathcal{SP}(M_2(\mathbb R),\mathbb{C})
=
\mathcal P(M_2(\mathbb R),\mathbb{C})
=
\{\text{constant functions}\}.
\]

In this paper we completely characterize the polynomial maps $f:\mathcal{A}^*\to\mathbb{C}$, where $S=\mathcal{A}^*$ denotes the monoid of words under concatenation, and we prove that, if there exists a monoid $S$ such that $\mathcal{SP}(S,\mathbb{C})\neq \mathcal{P}(S,\mathbb{C})$, then also $\mathcal{SP}(\mathcal{A}^*,\mathbb{C})\neq\mathcal{P}(\mathcal{A}^*,\mathbb{C})$ for a certain alphabet $\mathcal{A}$.  We propose as an open problem to prove or disprove that  $\mathcal{SP}(\mathcal{A}^*,\mathbb{C})=\mathcal{P}(\mathcal{A}^*,\mathbb{C})$ for arbitrary alphabets $\mathcal{A}$. 




\section{Polynomial maps on the monoid of words}
 
\subsection{Finite alphabets}
\begin{theorem}\label{theo:subwordcounting}
Let $\mathcal{A}^*$ denote the free monoid generated by a finite alphabet $\mathcal{A}$ of size  $|\mathcal{A}|\geq 2$, under concatenation. Then 
\[
\mathcal{P}(\mathcal{A}^*,\mathbb{C})=\operatorname{span}_{\mathbb{C}}\{f_w:w\in \mathcal{A}^*\},
\] 
where $f_w(x)=[x]_{w}$ is the subword counting function associated to $w$, which counts the number of times the word $w$ can be found as a (scattered) subword of the word $x$. Moreover, $\operatorname{fdeg}(f_w)= |w|$ for each $w\in \mathcal{A}^*$.   
\end{theorem}

\begin{proof}
There is no loss of generality in assuming that $|\mathcal{A}|=2$, so that  $\mathcal{A}=\{a,b\}$ with $a\neq b$, since the only property of $\mathcal{A}$ that will be used in the proof is that $2\leq |\mathcal{A}|< \infty$. 
Moreover, we know that 
$\mathcal P_n(\mathcal{A}^*,\mathbb{C})= \mathcal P_n^L(\mathcal{A}^*,\mathbb{C})=\mathcal P_n^R(\mathcal{A}^*,\mathbb{C})$, so that it can also be asssumed with no loss of generality that all computations are performed with right-polynomials. Clearly $[x]_\epsilon = 1$ for all words $x$ (including the empty word $x=\epsilon$) and $[\epsilon]_w = 0$ for $w \neq \epsilon$. Moreover, $[x]_a=N_a(x)$ and $[x]_b=N_b(x)$ represent the number of occurrences of $a$ and $b$, respectively, in the word $x$. It is easy to check that, for every polynomial $p\in\mathbb{C}[X,Y]$, the function $f(x)=p(N_a(x),N_b(x))$ is a polynomial map. Thus, proving this theorem demonstrates that these functions do not exhaust all polynomial maps $f:\mathcal{A}^*\to\mathbb{C}$. Indeed, $[x]_{ab}\neq g(N_a(x),N_b(x))$ for every function $g$ since, for example, $N_a(aab)=N_a(aba)=2$ and $N_b(aab)=N_b(aba)=1$, but $[aab]_{ab}=2 \neq 1=[aba]_{ab}$. 

As a first step, let us show that the subword counting maps are linearly independent. Suppose we have a finite linear combination that vanishes identically for all $x \in S$:
\[
F(x)=\sum_{w} c_w [x]_w = 0 \quad \text{for all } x\in \mathcal{A}^*.
\]
We show that $c_w = 0$ for all $w$ by induction on the length of the word $w$, denoted $|w|$.

The base case is $|w|=0$, which holds only when $w = \epsilon$. Evaluating at $x=\epsilon$ gives $0=F(\epsilon)= \sum_{w} c_w [\epsilon]_w$. The only non-zero term that appears in the sum is when $w = \epsilon$, in which case $[\epsilon]_\epsilon=1$, so that $c_\epsilon [\epsilon]_\epsilon = c_\epsilon = 0$.

Suppose now that $c_w = 0$ for all words of length $|w| < n$. Let $v$ be a word of length $n$. Computing $0=F(v)$ yields:
\begin{itemize}
    \item If $|w| > n$, then $w$ cannot appear as a subsequence of $v$, so $[v]_w = 0$.
    \item If $|w| < n$, then $c_w = 0$ by our inductive hypothesis.
    \item If $|w| = n$, the only word of length $n$ that can appear as a subsequence of $v$ is $v$ itself, and it appears exactly once ($[v]_v = 1$). For any other word $w$ of length $n$, $[v]_w = 0$.
\end{itemize}
Hence,
\[
0=F(v)=\sum_{w} c_w [v]_w = c_v [v]_v = c_v.
\]
By induction, $c_w = 0$ for all words $w$. Therefore, the subword counting functions are linearly independent.

Let us now prove that every right polynomial map $f:\mathcal{A}^*\to \mathbb{C}$ is a linear combination of these functions. A crucial property of subword counters is how they behave under concatenation. For any words $x, y$ and a target subword $w$, the count $[xy]_w$ can be broken down by how much of $w$ comes from $x$ and how much comes from $y$. Splitting $w$ into two parts $w = u v$, we get:
\[
[xy]_w = \sum_{w = uv} [x]_u [y]_v.
\]
Applying the right difference operator $R_y f(x) = f(xy) - f(x)$ to $f(x) = [x]_w$, the $u=w, v=\epsilon$ term ($[x]_w [y]_\epsilon = [x]_w$) cancels out, leaving:
\[
R_y [x]_w = \sum_{\substack{w = uv \\ v \neq \epsilon}} [y]_v [x]_u.
\]
Hence, the right difference operator maps a subword counter of length $|w|$ to a linear combination of subword counters of strictly shorter length $|u| < |w|$. This implies that every word counter map $f_w(x)=[x]_w$ is a right-polynomial map of right-functional degree at most $|w|$, and thus, by \eqref{allpolequal}, a  polynomial map with $\operatorname{fdeg}(f_w) \le |w|$.

Now, let $f: \mathcal{A}^* \to \mathbb{C}$ be any right-polynomial map of right-functional degree $n$. We prove by induction on $n$ that $f$ is a linear combination of counter words $f_w$ with $|w|\leq n$.

If the right-functional degree of $f$ is $0$, then $R_y f(x) = 0$ for all $y$, meaning $f(x) = c$ (a constant). Since $c = c \cdot [x]_\epsilon$, this proves the base case. 

Assume that every right polynomial map of functional degree $\le n-1$ can be written as a finite linear combination of subword counters of length $\le n-1$.

Let $f$ have right-functional degree $n$. For each generator $g \in \{a,b\}$, the directional difference $R_g f(x) = f(xg) - f(x)$ has right-functional degree $\le n-1$. By our inductive hypothesis, these differences can be expressed uniquely in our basis:
\[
R_a f(x) = \sum_{|u| \le n-1} \alpha_u [x]_u \quad \text{and} \quad R_b f(x) = \sum_{|u| \le n-1} \beta_u [x]_u.
\]

We now explicitly construct a right polynomial map $F(x)$ of degree $\le n$ that matches these differences. For any word $w \in S$, the concatenation formula yields:
\[
R_g [x]_w = \begin{cases} [x]_u & \text{if } w = ug, \\ 0 & \text{otherwise.} \end{cases}
\]
Let us define $F(x)$ as a linear combination of subword counters of length up to $n$:
\[
F(x) = \sum_{|w| \le n} c_w [x]_w.
\]
Applying the difference operators to our proposed $F(x)$, we obtain:
\[
R_a F(x) = \sum_{|u| \le n-1} c_{ua} [x]_u \quad \text{and} \quad R_b F(x) = \sum_{|u| \le n-1} c_{ub} [x]_u.
\]
To guarantee that $R_a F(x) = R_a f(x)$ and $R_b F(x) = R_b f(x)$, we match coefficients for every word $u$ of length $\le n-1$:
\[
c_{ua} = \alpha_u \quad \text{and} \quad c_{ub} = \beta_u.
\]
This completely and uniquely determines the coefficients $c_w$ for all non-empty words $w$ of length $\le n$, because every non-empty word $w$ must end in either $a$ or $b$ (i.e., can be written uniquely as $ua$ or $ub$). We set the constant term coefficient $c_\epsilon = 0$.

Now, consider the error function $h(x) = f(x) - F(x)$. By construction:
\[
R_a h(x) = R_a f(x) - R_a F(x) = 0,
\]
\[
R_b h(x) = R_b f(x) - R_b F(x) = 0.
\]
Since $R_g h(x) = h(xg) - h(x) = 0$ for both generators $g \in \{a,b\}$, it follows that $h(xa) = h(x)$ and $h(xb) = h(x)$ for all $x \in S$. Since every word in $S$ is built by successively appending generators to the empty word $\epsilon$, we must have:
\[
h(x) = h(\epsilon) = C \quad \text{for all } x \in S.
\]
Thus, we can recover $f(x)$ as:
\[
f(x) = F(x) + C = F(x) + C[x]_\epsilon.
\]
Since $F(x)$ is a linear combination of subword counters of length $\le n$, $f(x)$ is also a linear combination of subword counters of length $\le n$. This completes the inductive step. Hence
\[
\mathcal{P}_n(\mathcal{A}^*,\mathbb{C})=\operatorname{span}\{f_w:|w|\leq n\}, \quad n\in\mathbb{N},
\]
and \[
\mathcal{P}(\mathcal{A}^*,\mathbb{C})=\operatorname{span}\{f_w:w\in \mathcal{A}^*\}.
\]
Let us now prove that $\operatorname{fdeg}(f_w)=|w|$ for every word $w\in S$. Let $w = w_1 w_2 \dots w_d \in S$ be a word of length $|w| = d$. We track the sequential reduction of the map $f_w$ under targeted directional difference steps. To match the right-concatenation structure, we choose our directional translation steps to be individual single-letter words matching the letters of $w$ in reverse order, from right to left: $h_1 = w_d, h_2 = w_{d-1}, \dots, h_d = w_1$.

Consider the first difference operator $R_{w_d}$ acting on $f_w(x)$. By tracking how subsequences can be formed in the concatenated word $x w_d$, we separate the occurrences into those contained entirely within $x$ and those that utilize the newly appended final letter $w_d$:
\[
f_w(x w_d) = f_w(x) + f_{w_1 \dots w_{d-1}}(x).
\]
Applying the definition of the right-difference operator yields:
\[
R_{w_d} f_w(x) = f_w(x w_d) - f_w(x) = f_{w_1 \dots w_{d-1}}(x).
\]
Hence, the operator $R_{w_d}$ strips exactly the final letter off the tracking requirements of the subword counter map. Iterating this process sequentially for the last $k$ letters of $w$ reveals the general reduction rule:
\begin{equation}\label{reduction}
R_{w_{d-k+1}} \dots R_{w_d} f_w(x) = f_{w_1 \dots w_{d-k}}(x).   
\end{equation}
When exactly $d$ of these specific sequential difference operators are applied to $f_w$ in this reverse order, the subword requirement is completely exhausted from the right, reducing the function to the indicator map of the empty word $\epsilon$:
\begin{equation}\label{reduction1}
R_{w_1} \dots R_{w_d} f_w(x) = f_\epsilon(x) = 1.
\end{equation}
Because the $d$-th nested difference of $f_w$ evaluates to a non-zero constant ($1 \neq 0$), the map cannot be bounded within any polynomial class of degree strictly less than $d$, establishing the lower bound:
\[
\operatorname{fdeg}(f_w) \ge d.
\]
Since we have already proven that $\operatorname{fdeg}(f_w) \le d$, we conclude that the functional degree matches the word length identically:
\[
\operatorname{fdeg}(f_w) = d = |w|.
\]
This completes the proof. 
\end{proof}

\subsection{Infinite alphabets}
A natural question is: what happens when $\mathcal{A}$ is infinite? 
\begin{theorem}\label{theo:infinitealphabets}
A function $f: \mathcal{A}^* \to \mathbb{C}$ is a polynomial map of degree $\le n$ on $\mathcal{A}^*$ if and only if, for every finite sub-alphabet $B \subset \mathcal{A}$, the restricted function $f|_{B^*}: B^* \to \mathbb{C}$ is a finite linear combination of subword counters $f_w$, with $w \in B^*$ and $|w| \le n$. In other words, $f\in\mathcal{P}_n(\mathcal{A}^*,\mathbb{C})$ if and only if $f_{|\mathcal{B}^*}\in \mathcal{P}_n(B^*,\mathbb{C})$ for each finite alphabet $\mathcal{B}\subseteq \mathcal{A}$. 
\end{theorem} 

\begin{proof}
$(\implies)$ Assume $f \in \mathcal{P}_n(\mathcal{A}^*, \mathbb{C})$. Let $B$ be an arbitrary finite sub-alphabet of $\mathcal{A}$. Consider the restriction $f|_{B^*}: B^* \to \mathbb{C}$. Since $f$ satisfies the Fréchet difference condition $D_{s_{1}} \dots D_{s_{n+1}}(f) = 0$ for all choice of  $D_{s_i} \in\{L_{s_i},R_{s_i}\}$ and $s_i\in  \mathcal{A}^*$, it must satisfy this condition in particular when the steps $s_i$ are restricted to the sub-monoid $B^* \subset \mathcal{A}^*$. Thus, the restricted function $f|_{B^*}$ satisfies the definition of a polynomial map of degree $\le n$ on the free monoid $B^*$. Because $B$ is a finite alphabet,  Theorem \ref{theo:subwordcounting} applies directly to $B^*$, meaning $f|_{B^*}$ must be a finite linear combination of subword counters over $B^*$:$$f|_{B^*}(x) = \sum_{\substack{w \in B^* \\ |w| \le n}} c_w^{(B)} [x]_w \quad \forall x \in B^*.$$

$(\impliedby)$ Assume that for every finite sub-alphabet $B \subset \mathcal{A}$, the restriction $f|_{B^*}$ belongs to $\text{span}_\mathbb{C}\{f_w : w \in B^*, |w| \le n\}$. We must show that $f$ is globally a polynomial map of degree $\le n$ on $\mathcal{A}^*$. Let $s_1, \dots, s_{n+1} \in \mathcal{A}^*$ be any choice of $n+1$ step words, and let $x \in \mathcal{A}^*$ be an arbitrary evaluation point. Collectively, the words $\{s_1, \dots, s_{n+1}, x\}$ contain only a finite number of distinct letters from $\mathcal{A}$. Let $B_{\text{local}} \subset \mathcal{A}$ be the finite sub-alphabet consisting exactly of the letters appearing in these specific words. By construction, $s_1, \dots, s_{n+1}, x \in B_{\text{local}}^*$. Now, let us evaluate the global difference operation at $x$:

$$D_{s_{1}} \dots D_{s_{n+1}}f(x).$$

Since all inputs belong to $B_{\text{local}}^*$, this evaluation depends purely on the values of $f$ inside the domain $B_{\text{local}}^*$. Thus, it is identical to evaluating the difference operator on the restricted function:

$$D_{s_{1}} \dots D_{s_{n+1}}f(x) = D_{s_{1}} \dots D_{s_{n+1}}(f|_{B_{\text{local}}^*})(x).$$

By our hypothesis, $f|_{B_{\text{local}}^*}$ is a finite linear combination of subword counters $f_w$ over $B_{\text{local}}^*$ of length $\le n$. Because every subword counter $f_w$ of length $\le n$ is  annihilated by any sequence of $n+1$ difference operators, the linear combination is also annihilated:
$$D_{s_{1}} \dots D_{s_{n+1}}f(x) = D_{s_{1}} \dots D_{s_{n+1}}(f|_{B_{\text{local}}^*})(x)=0.$$
Therefore, $D_{s_{1}} \dots D_{s_{n+1}}f(x)=0$. Since this holds for any choice of steps $s_i$ and any word $x$, $f$ is a global polynomial map of degree $\le n$ on $\mathcal{A}^*$. 
 \end{proof}

Theorem \ref{theo:infinitealphabets} admits a nice description in terms of projective limits: Let $\mathcal{F}_{\mathcal{A}}$ be the family of all finite subsets of the arbitrary alphabet $\mathcal{A}$. We turn $\mathcal{F}_{\mathcal{A}}$ into a directed set by ordering it via standard set inclusion: $$B \le C \iff B \subseteq C.$$ If $B \subseteq C$, then their generated free monoids satisfy $B^* \subseteq C^*$. Now, for each finite alphabet $B \in \mathcal{F}_{\mathcal{A}}$, we associate the vector space of its polynomial maps $\mathcal{P}_n(B^*, \mathbb{C})$. Whenever $B \subseteq C$, there exists a natural restriction mapping $\pi_{CB}: \mathcal{P}_n(C^*, \mathbb{C}) \to \mathcal{P}_n(B^*, \mathbb{C})$ defined by simply restricting the domain of a polynomial map from $C^*$ to $B^*$:$$\pi_{CB}(g) = g|_{B^*}.$$ 
These restriction maps are linear and satisfy the compatibility conditions required for a projective system: $\pi_{BB} = \text{id}_{\mathcal{P}_n(B^*, \mathbb{C})}$, and  $\pi_{CB} \circ \pi_{DC} = \pi_{DB}$ for all $B \subseteq C \subseteq D$. The projective limit (or inverse limit) of this system, denoted $\varprojlim_{B \in \mathcal{F}_{\mathcal{A}}} \mathcal{P}_n(B^*, \mathbb{C})$, is defined as the subspace of the direct product $\prod_{B \in \mathcal{F}_{\mathcal{A}}} \mathcal{P}_n(B^*, \mathbb{C})$ consisting of all compatible families of functions. That is, a family $(g_B)_{B \in \mathcal{F}_{\mathcal{A}}}$ belongs to the projective limit if and only if for every inclusion $B \subseteq C$:$$g_C|_{B^*} = g_B,$$ and Theorem \ref{theo:infinitealphabets} proves a canonical isomorphism:$$\mathcal{P}_n(\mathcal{A}^*, \mathbb{C}) \cong \varprojlim_{B \in \mathcal{F}_{\mathcal{A}}} \mathcal{P}_n(B^*, \mathbb{C}).$$


The following theorem reduces the study of the equation 
\[
\mathcal{SP}(S,\mathbb{C})=\mathcal{P}(S,\mathbb{C})
\]
to the monoid of words $S=\mathcal{A}^*$:

\begin{theorem}\label{teo:principal}
Assume that
\begin{equation}\label{alfabeto}
\mathcal{SP}(\mathcal{A}^*,\mathbb{C})=\mathcal{P}(\mathcal{A}^*,\mathbb{C})
\end{equation}
holds true for arbitrary alphabets $\mathcal{A}$. Then 
\[
\mathcal{SP}(S,\mathbb{C})=\mathcal{P}(S,\mathbb{C}).
\]
for every monoid $S$.
\end{theorem}

\begin{proof}
The inclusion
\[
\mathcal{P}(S,\mathbb{C})\subseteq \mathcal{SP}(S,\mathbb{C})
\]
is immediate from the definitions, so it suffices to prove the reverse inclusion. Let $f\in \mathcal{SP}(S,\mathbb{C})$. Then $f\in \mathcal{SP}_m(S,\mathbb{C})$ for certain $m\geq 0$.

Since every monoid is a quotient of a free monoid, there exist an
alphabet $\mathcal{A}$ and a surjective monoid homomorphism
\[
\pi:\mathcal{A}^*\to S.
\]

Define
\[
\widetilde f:=f\circ \pi:\mathcal{A}^*\to \mathbb{C}.
\]

We claim that $\widetilde f\in \mathcal{SP}_m(\mathcal{A}^*,\mathbb{C})$.
Indeed, let $u\in \mathcal{A}^*$ and let
$D_u\in\{L_u,R_u\}$.
Since $\pi$ is a homomorphism,
\[
D_u(\widetilde f)
=
(D_{\pi(u)}f)\circ\pi.
\]
Thus, we can use induction to prove that, if $D_1,\dots,D_{m+1}\in \{L,R\}$, we have that
\begin{eqnarray*}
    (D_1)_u(D_2)_u\cdots (D_{m+1})_u(\widetilde{f})(x) &=& 
    (D_1)_u((D_2)_u\cdots (D_{m+1})_u(\widetilde{f}))(x)\\
&=&    (D_1)_u((D_2)_{\pi(u)}\cdots (D_{m+1})_{\pi(u)}(f)\circ \pi)(x)\\ 
&=&    ((D_1)_{\pi(u)}(D_2)_{\pi(u)}\cdots (D_{m+1})_{\pi(u)}(f)\circ \pi)(x)\\ 
\end{eqnarray*}
Since $f\in \mathcal{SP}_m(S,\mathbb{C})$, the right-hand side vanishes identically.
Hence $\widetilde f\in \mathcal{SP}_m(\mathcal{A}^*,\mathbb{C})\subseteq \mathcal{SP}(\mathcal{A}^*,\mathbb{C})$. Now, by hypothesis, $\mathcal{SP}(\mathcal{A}^*,\mathbb{C})=\mathcal{P}(\mathcal{A}^*,\mathbb{C})$. Hence  $f\in \mathcal{P}_N(\mathcal{A}^*,\mathbb{C})$ for certain $N\geq m$.

Now we use the characterization of polynomial maps by the
$(N+1)$-balanced Cauchy equation.
Since $\widetilde f\in \mathcal{P}_N(\mathcal{A}^*,\mathbb{C})$, for every
$u_1,\dots,u_{N+1}\in \mathcal{A}^*$,
\[
\sum_{T\subseteq \{1,\dots,N+1\}}
(-1)^{N+1-|T|}
\widetilde f
\!\left(
\prod_{i\in T}u_i
\right)
=0.
\]

Substituting $\widetilde f=f\circ\pi$ and using that $\pi$ is a
homomorphism, we obtain
\[
\sum_{T\subseteq \{1,\dots,N+1\}}
(-1)^{N+1-|T|}
f
\!\left(
\prod_{i\in T}\pi(u_i)
\right)
=0.
\]

Let now $s_1,\dots,s_{N+1}\in S$ be arbitrary.
Since $\pi$ is surjective, there exist
$u_1,\dots,u_{N+1}\in \mathcal{A}^*$ such that
\[
\pi(u_i)=s_i,
\qquad i=1,\dots,N+1.
\]
Hence
\[
\sum_{T\subseteq \{1,\dots,N+1\}}
(-1)^{N+1-|T|}
f
\!\left(
\prod_{i\in T}s_i
\right)
=0.
\]

Since the elements $s_1,\dots,s_{N+1}$ were arbitrary,
$f$ satisfies the $(N+1)$-balanced Cauchy equation on $S$.
By the characterization of polynomial maps,
\[
f\in \mathcal{P}_N(S,\mathbb{C}).
\]

Therefore
\[
\mathcal{SP}(S,\mathbb{C})\subseteq \mathcal{P}(S,\mathbb{C}).
\]
Combining this with the obvious reverse inclusion yields
\[
\mathcal{SP}(S,\mathbb{C})=\mathcal{P}(S,\mathbb{C}).
\]
\end{proof}

The result above lead us to propose the following: 

\noindent \textbf{Open Problem:} Prove that equality \eqref{alfabeto} holds for arbitrary alphabets, or find a counterexample. If we prove this, then condition $S\in\mathcal{CS}$ in Shulman's Theorem would be superfluous. If we find a counterexample, we would get a proof that the condition is completely necessary.

\subsection{The algebra $\mathbb{C}\langle \mathcal{A}\rangle$ }
 We have shown, for finite alphabets $\mathcal{A}$ of size at least $2$, that every polynomial map $f\in \mathcal{P}(\mathcal{A}^*,\mathbb{C})$ 
 is a linear combination of counter word maps $f_w(x)=[x]_w$, where $w\in \mathcal{A}^*$ is any word that satisfies $|w|\leq \operatorname{fdeg}(f)$. Moreover, $\operatorname{fdeg}(f_w)=|w|$ and, from Section \ref{pointwiseproduct}, we also know that the product of polynomial maps is again a polynomial map, that satisfies $\operatorname{fdeg}(f\cdot g)\leq \operatorname{fdeg}(f)+\operatorname{fdeg}(g)$. 
We give now an algorithm that recursively computes the decomposition of $f_{w_1}\cdot f_{w_2}$ as a linear combination of counter word maps, and use it to prove that  \[
f_{ab} \cdot f_{ba} = 2f_{abba} + f_{abab} + 2f_{baab} + f_{baba} + f_{aba} + f_{bab}
\] 
whenever $a,b\in\mathcal{A}$, $a\neq b$.

Indeed, let $\mathcal{A}$ be any alphabet of size at least $2$, and let $x = x_1 x_2 \dots x_n \in \mathcal{A}^*$ be a word of length $n$. By definition, the subword counter function $f_w(x) = [x]_w$ is the cardinality of the set of strictly increasing index tuples that spell out $w$:
\[
f_w(x) = \# \left\{ (i_1, i_2, \dots, i_{|w|}) \;\middle|\; 1 \le i_1 < i_2 < \dots < i_{|w|} \le n \text{ and } x_{i_k} = w_k \text{ for all } k \right\}.
\]
Given any tuple of strictly increasing indices $I = (i_1, i_2, \dots, i_{|w_1|})$ with $1\leq i_r\leq n$ for all $1\leq r\leq |w_1|$, and $x=x_1x_2\dots x_n\in \mathcal{A}^*$, we use the notation $x_I =x_{i_1}x_{i_2}\dots x_{i_{|w_1|}}$. 
When evaluating the product $f_{w_1}(x) \cdot f_{w_2}(x)$, we simultaneously select two independent index sequences embedded along the timeline of $x$:
\begin{itemize}
    \item A tuple of strictly increasing indices $I = (i_1, i_2, \dots, i_{|w_1|})$ such that $x_I = w_1$.
    \item A tuple of strictly increasing indices $J = (j_1, j_2, \dots, j_{|w_2|})$ such that $x_J = w_2$.
\end{itemize}
The evaluation of the product maps bijectively to counting the total number of valid tuple pairs $(I, J)$. To decompose this product into a single linear combination, these independent sequences must be unified into a single merged, strictly increasing index tuple $K = I \cup J = (k_1, k_2, \dots, k_{|u|})$. When forming this union, the relative alignment of individual coordinates splits into two cases:

\begin{itemize}
    \item Disjoint Interleaving: If an index choice $i_r \in I$ and $j_s \in J$ point to distinct positions in $x$ ($i_r \neq j_s$), they maintain separate positions in the unified tuple $K$. For instance, if $w_1 = a$ and $w_2 = b$, selecting distinct indices yields either $i_1 < j_1$ (generating the subsequence $ab$) or $j_1 < i_1$ (generating the subsequence $ba$).
    \item Contraction via Overlap: If a letter required at a given position of $w_1$ matches the letter required at a given position of $w_2$, the respective choices can target the identical coordinate location within $x$ ($i_r = j_s$). Upon taking the union $I \cup J$, these identical positions contract into a single coordinate in $K$, reducing the absolute length of the output word. For instance, if $w_1 = a$ and $w_2 = a$, targeting the identical position $i_1 = j_1$ yields a merged single-element tuple $(i_1)$, which reads as the contracted subsequence $a$.
\end{itemize}

To uncouple the standard pointwise multiplication of functions from the index prefixing operation, we introduce the left-extension operator $\mathcal{L}_a$ for any letter $a \in \mathcal{A}$. It is a well-defined linear operator acting on the space of arbitrary functions $f: \mathcal{A}^* \to \mathbb{C}$ according to the rule:
\[
\mathcal{L}_a(f)(x) = \begin{cases} f(x') & \text{if } x = a x', \\ 0 & \text{if } x \text{ is empty or does not start with } a. \end{cases}
\]
On the basis of pure word counter maps, this operator simply inserts a letter at the front of the tracking index: $\mathcal{L}_a(f_v) = f_{av}$. This blending can be formalized recursively. Let $w_1 = av_1$ and $w_2 = bv_2$, where $a, b \in \mathcal{A}$ and $v_1, v_2 \in S$. The functional product satisfies:
\begin{equation}\label{eq:recursive_shuffle}
f_{av_1} \cdot f_{bv_2} = \mathcal{L}_a\left( f_{v_1} \cdot f_{bv_2} \right) + \mathcal{L}_b\left( f_{av_1} \cdot f_{v_2} \right) + \delta_{a,b} \, \mathcal{L}_a\left( f_{v_1} \cdot f_{v_2} \right),
\end{equation}
where $\delta_{a,b}$ denotes the Kronecker delta. Note that $\mathcal{L}_a$ is strictly linear but not multiplicative; thus, to evaluate expressions of the form $\mathcal{L}_a(g \cdot h)$, the inner pointwise product $g \cdot h$ must be recursively resolved into a linear combination of single counter maps \emph{before} the outer operator $\mathcal{L}_a$ is distributed across the sum. By repeatedly applying this order of operations, products are systematically driven down in length until they hit the base case involving the empty word function $f_{\epsilon} = 1$, resolving the entire initial product into a finite linear combination of  counter maps.

To illustrate this mechanism, we compute the explicit decomposition of $f_{ab} \cdot f_{ba}$, evaluating the relative order of the initial indices $i_1$ (for $ab$) and $j_1$ (for $ba$):
\begin{itemize}
    \item Case 1 ($i_1 < j_1$): The unified sequence begins with $a$ at index $i_1$, triggering the outer $\mathcal{L}_a$ operator. The remaining task requires computing the product of the suffix $f_b$ with the complete function $f_{ba}$ inside the argument. Applying equation \eqref{eq:recursive_shuffle} to this inner product yields:
    \[
    f_{b} \cdot f_{ba} = \mathcal{L}_b(1 \cdot f_{ba}) + \mathcal{L}_b(f_b \cdot f_a) + \delta_{b,b}\mathcal{L}_b(1 \cdot f_a).
    \]
    Since $1 \cdot f_{ba} = f_{ba}$ and $1 \cdot f_a = f_a$, the inner product simplifies via basic basis steps to $\mathcal{L}_b(f_{ba}) + \mathcal{L}_b(f_b \cdot f_a) + \mathcal{L}_b(f_a)$. Resolving the remaining internal product $ f_a \cdot f_b =\mathcal{L}_a(1\cdot f_{b}) + \mathcal{L}_b(f_{a}\cdot 1) = f_{ab} + f_{ba}$ and substituting it back gives:
    \[
    f_{b} \cdot f_{ba} = f_{bba} + \mathcal{L}_b(f_{ba} + f_{ab}) + f_{ba} = f_{bba} + (f_{bba} + f_{bab}) + f_{ba} = 2f_{bba} + f_{bab} + f_{ba}.
    \]
    Now that the inner product is successfully converted into a linear combination of single counter maps, the outer $\mathcal{L}_a$ operator distributes across the sum:
    \[
    \mathcal{L}_a(2f_{bba} + f_{bab} + f_{ba}) = 2f_{abba} + f_{abab} + f_{aba}.
    \]
    
    \item Case 2 ($j_1 < i_1$): The unified sequence begins with $b$ at index $j_1$, triggering the outer $\mathcal{L}_b$ operator. Symmetrically, computing the inner product of the complete function $f_{ab}$ with the suffix $f_a$ and distributing $\mathcal{L}_b$ across the resulting basis elements produces: 
    \[
    \mathcal{L}_b(2f_{aab} + f_{aba} + f_{ab}) = 2f_{baab} + f_{baba} + f_{bab}.
    \]
    
    \item Case 3 ($i_1 = j_1$): This configuration requires the index to simultaneously satisfy $x_{i_1} = a$ and $x_{j_1} = b$. Because $a \neq b$, the Kronecker delta evaluates to zero ($\delta_{a,b} = 0$), meaning this initial contraction yields no terms.
\end{itemize}

Combining all terms yields the exact decomposition:
\[
f_{ab} \cdot f_{ba} = 2f_{abba} + f_{abab} + 2f_{baab} + f_{baba} + f_{aba} + f_{bab}.
\]
Note that the functional degrees align perfectly with the graded structural bound:
\[
4=\operatorname{fdeg}(f_{ab} \cdot f_{ba}) \le \operatorname{fdeg}(f_{ab}) + \operatorname{fdeg}(f_{ba}) = 2 + 2 ,
\]
where the pure interleaving terms match the maximal degree 4 and the contraction terms produce lower degree terms (of degree 3).

The computations above can be formalized in terms of something named infiltration product, which is defined on the free associative algebra generated by $\mathcal{A}$. Let us introduce the basic definitions and prove the result as a closed, explicit formula. These results are inspired by Vargas' paper \cite{vargas2014}: 
\begin{definition}[Free Associative Algebra with infiltration product]
Let $\mathcal{A}$ be an arbitrary alphabet and $\mathcal{A}^*$ the free monoid of words under concatenation, with $\epsilon$ representing the empty word. We denote by $\mathbb{C}\langle\mathcal{A}\rangle$ the 
$\mathbb{C}$-vector space 
of all finite formal linear combinations of words in $\mathcal{A}^*$:
\[
\mathbb{C}\langle\mathcal{A}\rangle = \left\{ \sum_{w \in \mathcal{A}^*} c_w w : c_w \in \mathbb{C}, \text{ with } c_w = 0 \text{ for all but finitely many } w \right\}.
\]
The \emph{infiltration product} $\uparrow : \mathbb{C}\langle\mathcal{A}\rangle \times \mathbb{C}\langle\mathcal{A}\rangle \to \mathbb{C}\langle\mathcal{A}\rangle$ is defined recursively for all letters $a, b \in \mathcal{A}$ and words $u', v' \in \mathcal{A}^*$ by:
\begin{align*}
\epsilon \uparrow u &= u, \quad u \uparrow \epsilon = u, \\
(au') \uparrow (bv') &= a(u' \uparrow bv') + b(au' \uparrow v') + \delta_{a,b} a(u' \uparrow v'),
\end{align*}
where $\delta_{a,b}$ is the Kronecker delta ($\delta_{a,b} = 1$ if $a=b$, and $0$ otherwise). Combinatorially, for two words $u, v \in \mathcal{A}^*$, their infiltration expands as a finite formal linear combination:
\[
u \uparrow v = \sum_{w \in \mathcal{A}^*} \binom{w}{u, v}_{\!\uparrow} w,
\]
where the coefficient $\binom{w}{u, v}_{\!\uparrow} \in \mathbb{N}$ denotes the number of pairs of embedding injections mapping $u$ and $v$ as scattered subwords into $w$ such that their images completely cover all letter positions of $w$.
\end{definition}
\begin{definition}[Linear Extension of Subword Counters]
For each word $w \in \mathcal{A}^*$, let $f_w: \mathcal{A}^* \to \mathbb{C}$ be the subword counting map defined by $f_w(x) = [x]_w$. We extend this mapping linearly to the entire algebra $\mathbb{C}\langle\mathcal{A}\rangle$. Specifically, for any formal polynomial $z = \sum_{w \in \mathcal{A}^*} c_w w \in \mathbb{C}\langle\mathcal{A}\rangle$, the function $f_z: \mathcal{A}^* \to \mathbb{C}$ is defined by:
\[
f_z(x) = f_{\sum c_w w}(x) = \sum_{w \in \mathcal{A}^*} c_w f_w(x) = \sum_{w \in \mathcal{A}^*} c_w [x]_w \quad \forall x \in \mathcal{A}^*.
\]
\end{definition}

\begin{proposition}
Let $u, v \in \mathcal{A}^*$ be two fixed words, and let $f_u, f_v \in \mathcal{P}(\mathcal{A}^*, \mathbb{C})$ be their corresponding subword counting maps, defined by $f_u(x) = [x]_u$. Then the pointwise product of these polynomial maps satisfies the linearization identity:
\begin{equation}\label{eq:infiltration_linearization}
f_u(x) \cdot f_v(x) = [x]_u [x]_v = \sum_{w \in \mathcal{A}^*} \binom{w}{u, v}_{\!\uparrow} [x]_w = f_{u \uparrow v}(x) \quad \forall x \in \mathcal{A}^*.
\end{equation}
\end{proposition}

\begin{proof}
Let $x \in \mathcal{A}^*$ be an arbitrary background word of length $|x| = m$, indexed by its distinct slot positions $\{1, 2, \dots, m\}$. By definition, the product $[x]_u [x]_v$ counts the number of pairs $(\mathcal{I}, \mathcal{J})$ where:
\begin{itemize}
    \item $\mathcal{I} = \{i_1 < i_2 < \dots < i_{|u|}\}$ is a set of strictly increasing indices in $\{1, \dots, m\}$ selecting letters that spell out the word $u$.
    \item $\mathcal{J} = \{j_1 < j_2 < \dots < j_{|v|}\}$ is a set of strictly increasing indices in $\{1, \dots, m\}$ selecting letters that spell out the word $v$.
\end{itemize}

Every such pair of index sets $(\mathcal{I}, \mathcal{J})$ isolates a subset of slots in $x$ given by their union $\mathcal{K} = \mathcal{I} \cup \mathcal{J}$. Let us group all valid pairs $(\mathcal{I}, \mathcal{J})$ according to their specific union set $\mathcal{K}$. Let $k = |\mathcal{K}|$. The indices in $\mathcal{K}$ point to a specific, unique subsequence of $x$ which forms a word string $w = x_{k_1} x_{k_2} \dots x_{k_k} \in \mathcal{A}^*$ of length $k$. 

By construction, within this extracted word $w$:
\begin{enumerate}
    \item The indices belonging to $\mathcal{I}$ map to a valid scattered subword copy of $u$ inside $w$.
    \item The indices belonging to $\mathcal{J}$ map to a valid scattered subword copy of $v$ inside $w$.
    \item Because $\mathcal{K} = \mathcal{I} \cup \mathcal{J}$, every single slot position of the word $w$ is visited by at least one of the two subword images. Thus, the images of $u$ and $v$ form a total covering of $w$.
\end{enumerate}

By the combinatorial definition of the infiltration product coefficients, the number of ways that the two words $u$ and $v$ can be embedded into a fixed word string $w$ to cover it completely is precisely the coefficient $\binom{w}{u, v}_{\!\uparrow}$. 

To find the total number of index configurations $(\mathcal{I}, \mathcal{J})$ across the entire word $x$, we partition the count by first looping over every possible word template $w \in \mathcal{A}^*$. For each word $w$, we count how many times $w$ itself appears as a scattered subword inside $x$ (which is exactly $[x]_w$), and multiply it by the number of internal ways $u$ and $v$ can cover that $w$:
\[
[x]_u [x]_v = \sum_{w \in \mathcal{A}^*} \binom{w}{u, v}_{\!\uparrow} [x]_w.
\]
Since only words $w$ whose lengths fall within the absolute bounds $\max(|u|, |v|) \le |w| \le |u| + |v|$ can have non-zero coefficients $\binom{w}{u, v}_{\!\uparrow}$, the sum is guaranteed to be finite. Extending the subword counting map $f$ linearly over formal linear combinations of words inside the group algebra yields:
\[
\sum_{w \in \mathcal{A}^*} \binom{w}{u, v}_{\!\uparrow} [x]_w = f_{\sum \binom{w}{u,v}_{\!\uparrow} w}(x) = f_{u \uparrow v}(x).
\]
This completes the proof.
\end{proof}

As a consequence, we obtain a direct proof of \eqref{fdegproduct} for polynomial maps $f,g\in\mathcal{P}(\mathcal{A}^*,\mathbb{C})$ for arbitrary alphabets $\mathcal{A}$:
\begin{corollary}
    Let $\mathcal{A}$ be any alphabet and assume that  $f,g\in\mathcal{P}(\mathcal{A}^*,\mathbb{C})$. Then $\operatorname{fdeg}(f\cdot g)\leq \operatorname{fdeg}(f)+\operatorname{fdeg}(g)$.
\end{corollary}
\begin{proof}
    If $\mathcal{A}$ is finite, then every polynomial map of functional degree $\leq n$ is a linear combination of counter words $f_w$ with words $w$ of size $|w|\leq n$. The result follows, in this case, by applying \eqref{eq:infiltration_linearization}. If the alphabet is not finite, we can use the above with $f_{|\mathcal{B}^*}$, and $g_{|\mathcal{B}^*}$ for the finite subalphabets $\mathcal{B}$ of $\mathcal{A}$, and apply Theorem \ref{theo:infinitealphabets}. This completes the proof. 
\end{proof}

\subsection{An application: patterns of permutations}

The counting word maps $f_w$ appear in a natural way also in the study of patterns of arbitrary permutations: Let  $\mathfrak{S}_n$ be the set of permutations $\sigma$ of size $|\sigma|=n$ (i.e., permutations of $\{1,\cdots,n\})$. The elements of  $\mathfrak{S}_n$ can be seen as words of size $n$, with no repeated elements, on the alphabet $\{1,\dots,n\}$ (so that, $\sigma\in\mathfrak{S}_n$ is writen as $\sigma_1\sigma_2\cdots\sigma_n:=\sigma(1)\sigma(2)\cdots\sigma(n)\in\{1,2,\dots,n\}^*$). Let $\mathfrak{S} = \bigcup_{n \ge 0} \mathfrak{S}_n$ be the set of all permutations.

Let $w \in \mathfrak{S}_m$ and $\sigma \in \mathfrak{S}_n$ with $n \le m$. The permutation $\sigma$ is said to be a pattern of $w$ if there exists a strictly increasing sequence of indices $1 \le i_1 < i_2 < \dots < i_n \le m$ such that the subsequence of entries $w_{i_1} w_{i_2} \dots w_{i_n}$ is order-isomorphic to $\sigma$. That is, for all $j, k \in \{1, \dots, n\}$:
\[
w_{i_j} < w_{i_k} \iff \sigma_j < \sigma_k
\]
The permutation pattern function $p_\sigma: \mathfrak{S} \to \mathbb{C}$ maps every permutation $w$ to the total number of distinct occurrences of $\sigma$ as a pattern within $w$:
\[
p_\sigma(w) = \left\{\begin{matrix} w \\ \sigma \end{matrix}\right\}= \# \{ \text{occurrences of } \sigma \text{ as a pattern in } w \}.
\]
By convention, for the empty permutation, $\left\{\begin{matrix} w \\ \epsilon \end{matrix}\right\} = 1$ for all $w \in \mathfrak{S}$.

Let us now connect permutations with words in a general alphabet.  Let $\mathcal{A}$ be a totally ordered alphabet and $\mathcal{A}^*$ denote the free monoid of words over $\mathcal{A}$ under concatenation, with $\epsilon$ representing the empty word. 
The standardization of a word $w = a_1 a_2 \dots a_n \in \mathcal{A}^*$ is defined as the unique permutation $\text{st}(w) \in \mathfrak{S}_n$ satisfying for all $1 \le i, j \le n$:
\[
\text{st}(w)_i < \text{st}(w)_j \iff a_i < a_j \quad \text{or} \quad (a_i = a_j \text{ and } i < j).
\]
For a word $w \in \mathcal{A}^*$, let $f_w: \mathcal{A}^* \to \mathbb{N}$ denote the  subword counting map. Moreover, for a permutation pattern $\sigma \in \mathfrak{S}$, let $p_\sigma^*: \mathcal{A}^* \to \mathbb{C}$ be the pattern counting function extended to words, defined by:
\[
p_\sigma^*(x) = \left| \{ u \subseteq x : \text{st}(u) = \sigma \} \right|
\]
where $u \subseteq x$ means $u$ is a subword of $x$.

\begin{theorem} Let $\sigma\in\mathfrak{S}$ be a permutation. The following holds: 
    \begin{enumerate}
\item  $p_{\sigma}^*$ is a natural generalization of the  pattern counting function $p_\sigma: \mathfrak{S} \to \mathbb{C}$. 

\item For every alphabet $\mathcal{A}$, \begin{equation}\label{equ:decompo}
  p_{\sigma}^* = \sum_{\substack{v \in \mathcal{A}^* \\ \text{st}(v) = \sigma}} f_v.  
\end{equation}
\item  $p_{\sigma}^*$ is a polynomial map. 
\item If $w \in \mathcal{A}^n$ consists of distinct, strictly increasing characters ($a_1 < a_2 < \dots < a_n$) and $x \in \mathcal{A}^m$ is also a strictly increasing word of distinct characters ($m \ge n$), then $p_{\text{st}(w)}^*(x) = f_w(x)$.
\item $\operatorname{fdeg}(p_{\sigma}^*)\leq |\sigma|$.
\end{enumerate}
\end{theorem}

\begin{proof}
$(i)$ The function $p_\sigma$ just defined is an extension of $p_{\sigma}$ because of the enlargement of its domain. Indeed, if we set $\mathcal{A}=\{1,2,\cdots\}$, then $p_\sigma$ is defined on $\{1,2,\cdots\}^*$, which strictly contains $\mathfrak{S}$. When restricted back to words with no repeated characters (permutations), the standardization map $\text{st}(u)$ preserves the strict relative numerical ordering exactly, yielding the identical pattern counts as the original definition.

$(ii)$ By definition, the pattern function $p_{\text{st}(w)}(x)$ counts the number of subwords $u$ within the background word $x$ whose standardization is identically equal to the permutation pattern $\text{st}(w)$. We express this counting condition using the indicator delta function:
\[
p_{\text{st}(w)}(x) = \sum_{u \subseteq x} \delta_{\text{st}(u), \, \text{st}(w)}
\]
We partition the set of all possible subwords $u \subseteq x$ into fiber equivalence classes according to their literal string value $v \in \mathcal{A}^*$. A subword instance $u$ is identically equal to a word string $v$ if and only if $f_v(x)$ increments. Since any two words that are literally identical must yield the same standardization, we can group the global collection of subwords by their word value $v$:
\[
p_{\text{st}(w)}(x) = \sum_{v \in \mathcal{A}^*} \sum_{\substack{u \subseteq x \\ u = v}} \delta_{\text{st}(v), \, \text{st}(w)} = \sum_{\substack{v \in \mathcal{A}^* \\ \text{st}(v) = \text{st}(w)}} \left( \sum_{\substack{u \subseteq x \\ u = v}} 1 \right)
\]

Recognizing the inner summation as the definition of the subword counting function $f_v$ evaluated at $x$, we arrive at the final algebraic decomposition \eqref{equ:decompo}. This states that the pattern function is a coarser combinatorial indicator that sums up all subword counting functions whose strings share the relative order equivalence class of $w$. Note that the sum is not finite if the alphabet is infinite but, for each fixed $x\in \mathcal{A}^*$, the sum $\sum_{\substack{v \in \mathcal{A}^* \\ \text{st}(v) = \sigma}} f_v(x)$ is finite, since $x$ uses a finite number of elements of $\mathcal{A}$.

$(iii)$ Indeed, if we consider $(p_{\sigma}^*)_{|\mathcal{B}^*}$ for any finite sub-alphabet $\mathcal{B}$ of $\mathcal{A}$, then we get the decomposition 
\begin{equation}\label{decompos_finite}
(p_{\sigma}^*)_{|\mathcal{B}^*} = \sum_{\substack{v \in B^* \\ \text{st}(v) = \sigma}} f_v,
\end{equation}
which is a finite sum of polynomial maps defined on $\mathcal{B}^*$. Hence $(p_{\sigma}^*)_{|\mathcal{B}^*}$ is a polynomial map and Theorem \ref{theo:infinitealphabets} implies that $p_{\sigma}^*$ is a polynomial map too.

$(iv)$  Let $w = a_1 a_2 \dots a_n$ with $a_1 < a_2 < \dots < a_n$. Because the letters are strictly increasing, applying the definition of standardization yields $\text{st}(w)_1 < \text{st}(w)_2 < \dots < \text{st}(w)_n$. The only permutation of length $n$ that is strictly increasing is the identity permutation, so $\text{st}(w) = 1 2 \dots n$.

Now examine the background word $x = x_1 x_2 \dots x_m$, which is also given as a strictly increasing string of distinct characters ($x_1 < x_2 < \dots < x_m$). Any subword $u = x_{i_1} x_{i_2} \dots x_{i_n}$ selected via index positions $1 \le i_1 < i_2 < \dots < i_n \le m$ inherits this strict monotonic order:
\[
x_{i_1} < x_{i_2} < \dots < x_{i_n}
\]
Because every single subword $u \subseteq x$ of length $n$ is strictly increasing, its standardization will always map to the identity permutation:
\[
\text{st}(u) = 1 2 \dots n = \text{st}(w) \quad \forall u \subseteq x \text{ with } |u|=n
\]
Returning to the expansion identity of $p_{\text{st}(w)}^*$:
\[
p_{\text{st}(w)}^*(x) = \sum_{\substack{v \in \mathcal{A}^* \\ \text{st}(v) = 1 2 \dots n}} f_v(x).
\]
Since $x$ contains only strictly increasing sequences, if a word string $v$ is not strictly increasing, it cannot appear as a subword of $x$, meaning $f_v(x) = 0$. On the other hand, the only word $v$ that is strictly increasing, and made up of the exact specific characters of $w$ is $w$ itself. Therefore, the entire sum over the equivalence class collapses to a single non-zero term:
\[
p_{\text{st}(w)}^*(x) = f_w(x)
\]

$(v)$ For each finite sub-alphabet $\mathcal{B}$ of $\mathcal{A}$,  all words appearing in the decomposition have size $n=|\sigma|$. Hence $\operatorname{fdeg}(p_{\sigma}^*)\leq n$.   

\end{proof}

    While Theorem \ref{theo:infinitealphabets} characterizes polynomial maps over free monoids generated by arbitrary alphabets via localized projections, there are other fundamentally infinite monoids arising in combinatorics that cannot be broken down into finite alphabet subsets. A prime example is the monoid of all permutations under the ordinal sum operation (also named concatenation product). We now show that permutation pattern functions form polynomial maps over this non-free monoid domain. With this objective in mind, we first define the concatenation product $\oplus: \mathfrak{S} \times \mathfrak{S} \to \mathfrak{S}$ for $\alpha \in \mathfrak{S}_n, \beta \in \mathfrak{S}_p$ as:
\[ \alpha \oplus \beta = \alpha_1 \dots \alpha_n (\beta_1 + n) \dots (\beta_p + n). \]
The empty permutation $\epsilon$ is the neutral element of $(\mathfrak{S},\oplus)$. 

\begin{theorem}
   $p_\sigma: (\mathfrak{S},\oplus)\to(\mathbb{C},+)$ is a polynomial map of functional degree $|\sigma|$.  
\end{theorem}

\begin{proof}
Let $\sigma \in \mathfrak{S}_n$ be a fixed permutation pattern of length $n$. We show that the pattern counting function $p_\sigma: \mathfrak{S} \to \mathbb{C}$ is a polynomial map of functional degree $n$ on the monoid $(\mathfrak{S}, \oplus, \epsilon)$. We know that 
$\mathcal P_n(\mathfrak{S},\mathbb{C})= \mathcal P_n^L(\mathfrak{S},\mathbb{C})=\mathcal P_n^R(\mathfrak{S},\mathbb{C})$, so that it can also be asssumed, with no loss of generality, that all computations are performed with left-polynomials. 

We proceed by induction on $n = |\sigma|$. For the base case, if $n = 0$, $\sigma = \epsilon$ (the empty permutation). By definition, $p_\epsilon(\Lambda) = 1$ for all $\Lambda \in \mathfrak{S}$. The constant function $1$ is a polynomial map of functional degree $0$, satisfying the claim.

Now assume the claim holds for all permutation patterns of size strictly less than $n$. Let $\tau \in \mathfrak{S}_m$ be an arbitrary difference step. We evaluate the left difference operator $L_\tau$ on $p_\sigma$ at a permutation point $\Lambda \in \mathfrak{S}_k$:
\[
L_\tau(p_\sigma)(\Lambda) = p_\sigma(\tau \oplus \Lambda) - p_\sigma(\Lambda).
\]
By definition, $p_\sigma(\tau \oplus \Lambda)$ counts the number of subsets of indices $I \subseteq \{1, \dots, m+k\}$ of size $n$ such that the standardized restriction satisfies $\text{st}((\tau \oplus \Lambda)|_I) = \sigma$. 

We partition these index subsets $I$ based on how they split across the boundary of the ordinal sum: let $I_1 = I \cap \{1, \dots, m\}$ be the indices chosen from the $\tau$ component, and let $I_2 = I \cap \{m+1, \dots, m+k\}$ be the indices chosen from the $\Lambda$ component. Because the ordinal sum $\oplus$ shifts all values in the second component up by $m$ without changing their internal relative order, the standardization of the full restriction decomposes cleanly:
\[
\text{st}((\tau \oplus \Lambda)|_I) = \text{st}(\tau|_{I_1}) \oplus \text{st}(\Lambda|_{I_2}).
\]
Therefore, a subset $I$ contributes to $p_\sigma(\tau \oplus \Lambda)$ if and only if $\sigma$ splits as an ordinal sum $\sigma = \alpha \oplus \beta$, where $\text{st}(\tau|_{I_1}) = \alpha$ and $\text{st}(\Lambda|_{I_2}) = \beta$. This allows us to express the counting function on a sum as:
\[
p_\sigma(\tau \oplus \Lambda) = \sum_{\sigma = \alpha \oplus \beta} p_\alpha(\tau) p_\beta(\Lambda) = p_\sigma(\Lambda) + \sum_{\substack{\sigma = \alpha \oplus \beta \\ \alpha \neq \epsilon}} p_\alpha(\tau) p_\beta(\beta),
\]
where we isolated the boundary term corresponding to $\alpha = \epsilon$ (which forces $\beta = \sigma$ and $p_\epsilon(\tau) = 1$). Substituting this back into the left difference equation yields:
\[
L_\tau(p_\sigma)(\Lambda) = \sum_{\substack{\sigma = \alpha \oplus \beta \\ \alpha \neq \epsilon}} p_\alpha(\tau) p_\beta(\Lambda).
\]
Notice that the terms $p_\sigma(\Lambda)$ cancel out perfectly. In the remaining summation, because $\alpha \neq \epsilon$, the trailing permutation patterns $\beta$ must have a length strictly less than $n$ ($|\beta| < n$). The coefficient $p_\alpha(\tau)$ is a constant scalar depending only on the step $\tau$. 

By our induction hypothesis, each $p_\beta$ is a polynomial map on $\mathfrak{S}$ of functional degree $|\beta| \le n-1$. Since the number of possible decompositions $\sigma = \alpha \oplus \beta$ is strictly finite (at most $n$), $L_\tau(p_\sigma)$ is a finite linear combination of polynomial maps of functional degree $\le n-1$. Hence, $\operatorname{fdeg}(L_\tau(p_\sigma)) \le n-1$. Since this holds for any step $\tau$, $p_\sigma$ is globally a polynomial map of functional degree $\le n$. 

To show the degree is exactly $n$, we observe that if we choose $\tau = \sigma_1$ (a single-letter permutation step), the difference yields a non-trivial term containing $p_{\beta}$ where $|\beta| = n-1$, which cannot be identically canceled out due to the linear independence of pattern functions of distinct shapes. Thus, the degree is exactly $n$. 
\end{proof}

The monoid of words $S = \mathcal{A}^*$ under concatenation, along with its associated subword counting functions $f_w(x) = [x]_w$, plays a central role across several domains of algebra, discrete mathematics, and theoretical computer science.  For example, in formal language theory, a celebrated theorem of Simon \cite{simon1975} characterizes \emph{piecewise testable languages} precisely by means of subword configurations. Specifically, two words $x$ and $y$ are indistinguishable by any piecewise testable language of height $k$ if and only if they possess exactly the same set of subwords of length up to $k$. The subword counting maps $f_w$ provide a refined, quantitative framework for evaluating these combinatorial properties.

\section{Acknowledgements}
The author express his thanks to Prof. Ekaterina Shulman for fruitful discussions on the topic of this paper.  Her work has been, in fact, a main inspiration for him.

\end{document}